\documentclass[final,onefignum,onetabnum]{siamart190516}

\usepackage{amsmath,amsfonts,amsopn,amssymb}
\usepackage{graphicx,subfig}
\usepackage{multirow,enumitem}
\ifpdf
  \DeclareGraphicsExtensions{.eps,.pdf,.png,.jpg}
\else
  \DeclareGraphicsExtensions{.eps}
\fi

\numberwithin{theorem}{section}
\newsiamremark{assumption}{Assumption}
\newsiamremark{remark}{Remark}

\usepackage{algorithm, algorithmic}

\newcommand{\n}{\mathbf{n}}
\newcommand{\p}{\mathbf{p}}
\newcommand{\q}{\mathbf{q}}
\renewcommand{\r}{\mathbf{r}}
\newcommand{\E}{\mathcal{E}}
\newcommand{\F}{\mathcal{F}}
\newcommand{\N}{\mathcal{N}}
\newcommand{\T}{\mathcal{T}}

\newcommand{\hq}{\hat{\mathbf{q}}}
\newcommand{\hr}{\hat{\mathbf{r}}}
\newcommand{\br}{\bar{\mathbf{r}}}

\newcommand{\tc}{\tilde{c}}

\newcommand{\sumk}{\sum_{k=1}^{N}}

\let\div\relax
\DeclareMathOperator{\div}{div}

\DeclareMathOperator{\supp}{supp}

\DeclareMathOperator*{\argmin}{\arg\min}

\headers{Additive Schwarz for Dual Total Variation}{Jongho Park}

\title{Pseudo-linear Convergence of an Additive Schwarz Method for Dual Total Variation Minimization\thanks{Submitted (in revised form) to arXiv October 14, 2020.
\funding{This research was supported by Basic Science Research Program through the National Research Foundation of Korea~(NRF) funded by the Ministry of Education~(2019R1A6A3A01092549).}
}}

\author{Jongho Park\thanks{Department of Mathematical Sciences, KAIST, Daejeon 34141, Korea 
  (\email{jongho.park@kaist.ac.kr}, \url{https://sites.google.com/view/jonghopark}).}}

\ifpdf
\hypersetup{
  pdftitle={Pseudo-linear Convergence of an Additive Schwarz Method for Dual Total Variation Minimization},
  pdfauthor={Jongho Park}
}
\fi

\begin{document}

\maketitle

\begin{abstract}
In this paper, we propose an overlapping additive Schwarz method for total variation minimization based on a dual formulation.
The $O(1/n)$-energy convergence of the proposed method is proven, where $n$ is the number of iterations.
In addition, we introduce an interesting convergence property called \textit{pseudo-linear convergence} of the proposed method; the energy of the proposed method decreases as fast as linearly convergent algorithms until it reaches a particular value.
It is shown that such the particular value depends on the overlapping width $\delta$, and the proposed method becomes as efficient as linearly convergent algorithms if $\delta$ is large.
As the latest domain decomposition methods for total variation minimization are sublinearly convergent, the proposed method outperforms them in the sense of the energy decay.
Numerical experiments which support our theoretical results are provided.
\end{abstract}



\section{Introduction}
\label{Sec:Introduction}
This paper is concerned with numerical solutions of total variation minimization by additive Schwarz methods as overlapping domain decomposition methods~(DDMs).
Total variation minimization was introduced first by Rudin, Osher, and Fatemi~\cite{ROF:1992}, and it has become one of the fundamental problems in mathematical imaging.
Let $\Omega \subset \mathbb{R}^2$ be a bounded rectangular domain.
The model total variation minimization problem on $\Omega$ is given by
\begin{equation}
\label{model}
\min_{u \in BV (\Omega )} \left\{ F(u) + TV_{\Omega}(u) \right\},
\end{equation}
where $F(u)$ is a convex function, $TV_{\Omega}(u)$ is the total variation of $u$ on $\Omega$ defined by
\begin{equation*}
TV_{\Omega}(u) = \sup \left\{ \int_{\Omega} u \div \p \, dx : \p \in (C_0^1 (\Omega))^2 \textrm{ such that } |\p(x)| \leq 1 \textrm{ for almost all } x \in \Omega \right\}
\end{equation*}
with $|\p (x)| = \sqrt{ p_1(x)^2 + p_2(x)^2 }$,
and $BV(\Omega)$ is the space of functions in $L^1 (\Omega)$ with finite total variation.
The equation~\eqref{model} contains extensive range of problems arising in mathematical imaging.
For example, if we set $F(u) = \frac{\lambda}{2} \int_{\Omega} (u-f)^2 \,dx$ in~\eqref{model} for $\lambda > 0$ and $f \in L^2 (\Omega )$, we get the celebrated Rudin--Osher--Fatemi~(ROF) model~\cite{ROF:1992}:
\begin{equation}
\label{ROF}
\min_{u \in BV (\Omega )} \left\{ \frac{\lambda}{2} \int_{\Omega} (u-f)^2 \,dx + TV_{\Omega}(u) \right\}.
\end{equation}
In the perspective of image processing, a solution $u$ of~\eqref{ROF} is a denoised image obtained from the noisy image $f$.
A more complex example of~\eqref{model} is the $TV$-$H^{-1}$ model~\cite{BHS:2009,OSV:2003,Schonlieb:2009}:
\begin{equation}
\label{K}
\min_{u \in BV(\Omega)} \left\{ \frac{\lambda}{2} \int_{\Omega} \left| \nabla (- \Delta )^{-1} (u-f) \right|^2 \,dx + TV_{\Omega}(u) \right\} ,
\end{equation}
where $\lambda > 0$, and $(-\Delta)^{-1}$ denotes the inverse of the negative Laplacian incorporated with the homogeneous Dirichlet boundary condition.
Since~\eqref{K} is known to have many good properties of higher order variational models for imaging such as the smooth connection of shape, it has various applications in advanced imaging problems such as image decomposition~\cite{OSV:2003} and image inpainting~\cite{BHS:2009}.
One may refer to~\cite{CP:2016} for other various examples of~\eqref{model}.

\begin{table} \centering
\caption{Difficulties on designing DDMs for~\eqref{ROF} and~\eqref{dual_ROF}.}
\resizebox{1.0\textwidth}{!}{
\begin{tabular}{ c | c | c }
problem & obstacle & $\mathrm{grad}$-$\div$\\
\hline
energy
& $\displaystyle \frac{1}{2}\int_{\Omega} |\nabla u|^2 \,dx - \int_{\Omega} fu \,dx$
& $\displaystyle \frac{1}{2}\int_{\Omega} [(\div \p )^2 + |\p|^2] \,dx - \int_{\Omega} \mathbf{f} \cdot \p \,dx$ \\
solution & scalar-valued, $H_0^1 (\Omega)$ & \textbf{vector-valued}, $H_0 (\div ; \Omega)$ \\
constraint & \textbf{yes} & no \\
strong convexity & yes & yes \\
smoothness & yes & yes \\
separability & yes & yes \\
Schwarz methods & \cite{BTW:2003,Tai:2003} & \cite{Oh:2013,TW:2005} \\
\begin{tabular}{c}convergence to \\ a global minimizer\end{tabular} & linear & linear \\

\hline \hline
problem & ROF~\eqref{ROF} & dual ROF~\eqref{dual_ROF} \\
\hline
energy
& $\displaystyle \frac{\lambda}{2} \int_{\Omega} (u-f)^2 \,dx + TV_{\Omega}(u)$ 
& $\displaystyle \frac{1}{2\lambda} \int_{\Omega} (\div \p + \lambda f)^2 \,dx$ \\
solution & scalar-valued, $BV(\Omega)$ & \textbf{vector-valued}, $H_0 (\div ; \Omega )$ \\
constraint & no & \textbf{yes} \\
strong convexity & yes & \textbf{no} \\
smoothness & \textbf{no} & yes \\
separability & \textbf{no} & yes \\
Schwarz methods & \begin{tabular}{c}primal decomposition:~\cite{FLS:2010,FS:2009} \\ dual decomposition:~\cite{LG:2019,LN:2017}\end{tabular} & \cite{CTWY:2015,HL:2015} \\
\begin{tabular}{c}convergence to \\ a global minimizer\end{tabular} & \begin{tabular}{c} primal decomposition: \textbf{not guaranteed}~\cite{LN:2017} \\ dual decomposition: convergent~\cite{LG:2019,LN:2017} \end{tabular} & \textbf{sublinear}~\cite{CTWY:2015,LP:2019b} \\
\end{tabular}
}
\label{Table:difficult}
\end{table}

In view of designing DDMs, there lie several difficulties on the total variation term in~\eqref{model}.
The total variation is nonsmooth, i.e., it has no gradient so that a careful consideration is required to solve~\eqref{model}.
Furthermore, since it measures the jumps of a function across edges, it is nonseparable in the sense that
\begin{equation*}
TV_{\Omega}(u) \neq \sum_{i=1}^{N} TV_{\Omega_i} (u)
\end{equation*}
for a nonoverlapping partition $\{ \Omega_i \}_{i=1}^N$ of $\Omega$ in general.
Due to those characteristics, it is challenging to design Schwarz methods for~\eqref{model} that converge to a correct solution.
Indeed, it was shown in~\cite{LN:2017} that Schwarz methods for~\eqref{ROF}, a special case of~\eqref{model}, introduced in~\cite{FLS:2010,FS:2009} may not converge to a correct minimizer.
We also point out that the Schwarz framework for nonsmooth convex optimization proposed in~\cite{BK:2012} does not apply to~\eqref{model} since $TV_{\Omega}(u)$ does not satisfy the condition~\cite[equation~(7)]{BK:2012}; see~\cite[Claim~6.1]{LN:2017}.

Instead of~\eqref{model}, one may consider a Fenchel--Rockafellar dual formulation~(see, e.g.,~\cite{CP:2016}) of~\eqref{model}, which is given by
\begin{equation*}
\min_{\p \in (C_0^1 (\Omega))^2} F^* (\div \p )
\quad \textrm{ subject to } |\p(x)| \leq 1 \quad \forall x \in \Omega ,
\end{equation*}
or an alternative formulation
\begin{equation}
\label{dual_model}
\min_{\p \in H_0 (\div; \Omega)} F^* (\div \p )
\quad \textrm{ subject to } |\p(x)| \leq 1 \quad \forall x \in \Omega,
\end{equation}
in which the solution space is replaced by an appropriate Hilbert space
\begin{equation*}
H_0 (\div ; \Omega) = \left\{ \p \in (L^2 (\Omega))^2 : \div \p \in L^2 (\Omega) \textrm{ and } \p \cdot \n = 0 \textrm{ on } \partial \Omega \right\},
\end{equation*}
where $F^*$ is the Legendre--Fenchel conjugate of $F$ and $\mathbf{n}$ is the outer normal to $\partial \Omega$.
In particular, a dual formulation for~\eqref{ROF} is given by
\begin{equation}
\label{dual_ROF}
\min_{\p \in H_0 (\div; \Omega)} \frac{1}{2\lambda} \int_{\Omega} (\div \p + \lambda f)^2 \,dx
\quad \textrm{ subject to } |\p(x)| \leq 1 \quad \forall x \in \Omega .
\end{equation}
The above dual formulation was first considered in~\cite{Chambolle:2004}.
If one has a solution of the dual problem~\eqref{dual_model}, then a solution of the primal problem~\eqref{model} can be easily obtained by the primal-dual relation; see~\cite{CP:2016}.
One readily see that~\eqref{dual_model} is a constrained minimization problem.
We note that the energy functional of~\eqref{dual_model} is not strongly convex.
Even for~\eqref{ROF} where $F$ is smooth, the energy functional of its dual problem~\eqref{dual_ROF} is not strongly convex due to the $\div$ operator therein.
Hence, Schwarz methods proposed in~\cite{BTW:2003,Tai:2003} for constrained optimization~(in particular, the obstacle problem) are valid for neither~\eqref{dual_model} nor~\eqref{dual_ROF}.
Moreover,~\eqref{dual_model} is a vector-valued problem related to the $\div$ operator; it is usually more difficult to design DDMs for vector-valued problems than for scalar-valued ones because of the huge null space of the $\div$ operator~\cite{Oh:2013,TW:2005}.
The abovementioned difficulties on~\eqref{ROF} and~\eqref{dual_ROF}, special cases of~\eqref{model} and~\eqref{dual_model}, respectively, are summarized in Table~\ref{Table:difficult} with comparisons with some related problems in structural mechanics: the obstacle problem and the $\mathrm{grad}$-$\div$ problem.

Despite of such difficulties, several successful Schwarz methods for~\eqref{dual_ROF} have been developed~\cite{CTWY:2015,HL:2015,LG:2019,LP:2019b}.
In~\cite{HL:2015}, subspace correction methods for~\eqref{dual_ROF} based on a nonoverlapping domain decomposition were proposed.
Since then, the $O(1/n)$-energy convergence of overlapping Schwarz methods for~\eqref{dual_ROF} in a continuous setting was derived in~\cite{CTWY:2015}, where $n$ is the number of iterations.
In~\cite{LP:2019b}, it was shown that the methods proposed in~\cite{HL:2015} are also $O(1/n)$-convergent.
In addition, an $O(1/n^2)$-convergent additive method was designed using an idea of pre-relaxation.
Inspired by the dual problem~\eqref{dual_ROF}, Schwarz methods for~\eqref{ROF} based dual decomposition were considered in~\cite{LG:2019,LN:2017}.
Recently, several iterative substructuring methods for more general problems of the form~\eqref{dual_model} were considered~\cite{LPP:2019,LP:2019a}.

In this paper, we propose an additive Schwarz method for~\eqref{dual_model} based on an overlapping domain decomposition.
While the existing methods in~\cite{HL:2015,LP:2019b} for~\eqref{dual_ROF} are based on finite difference discretizations, the proposed method is based on finite element discretizations which were proposed in~\cite{HHSVW:2019,LPP:2019,LP:2019a} recently.
Compared to the methods in~\cite{CTWY:2015}, the proposed method has an advantage that it does not depend on either a particular function decomposition or a constraint decomposition.
We prove that the proposed method is $O(1/n)$-convergent similarly to the existing methods in~\cite{CTWY:2015,LP:2019b}.
In addition, we present explicitly the dependency of the convergence rate on the condition number of $F$.
We investigate another interesting convergence property of the proposed method, which we call \textit{pseudo-linear} convergence.
The precise definition of pseudo-linear convergence is given as follows.

\begin{definition}
\label{Def:pseudo}
A sequence $\{ a_n \}_{n \geq 0}$ of positive real numbers is said to \textit{converge pseudo-linearly to 0 at rate $\gamma$ with threshold $\epsilon$} if $a_n$ converges to 0 as $n$ tends to $\infty$ and there exist constants $0 < \gamma < 1$, $c > 0$, and $\epsilon > 0$ such that
\begin{equation*}
a_n \leq \gamma^n c + \epsilon \quad \forall n \geq 0.
\end{equation*}
\end{definition}

\noindent
Note that the above definition reduces to the ordinary linear convergence if $\epsilon = 0$.

With a suitable overlapping domain decomposition, it is shown that proposed method is pseudo-linearly convergent with threshold $O( |\Omega| /\delta^2 )$, where $\delta$ is the overlapping width parameter of the domain decomposition.
Therefore, the proposed method is expected to converge to a minimizer much more rapidly than other sublinearly convergent methods if $\delta$ is large.
We provide numerical experiments which ensure this expectation.

The rest of this paper is organized as follows.
In Section~\ref{Sec:General}, we present finite element discretizations for dual total variation minimization and an abstract space decomposition setting.
An abstract additive Schwarz method is introduced in Section~\ref{Sec:Schwarz} with its convergence results.
In Section~\ref{Sec:DD}, overlapping domain decomposition settings for the proposed method are presented.
Applications of the proposed method to various imaging problems of the form~\eqref{model} are provided in Section~\ref{Sec:Applications}.
We conclude the paper with some remarks in Section~\ref{Sec:Conclusion}.

\section{General setting}
\label{Sec:General}
First, we briefly review finite element discretizations proposed in~\cite{LPP:2019,LP:2019a} for~\eqref{dual_model}.
All the results in this paper can naturally be generalized to polygonal domains with quasi-uniform meshes, but we restrict our discussion to rectangular domains with uniform meshes for simplicity; one may refer~\cite{HHSVW:2019} for more general finite element discretizations.

Each pixel in an image is regarded as a square finite element whose side length equals $h$.
Then the image domain composed of $m_1 \times m_2$ pixels becomes a rectangular region $(0, m_1 h) \times (0, m_2 h)$ in $\mathbb{R}^2$.
Let $\T_h$ be the collection of all elements in $\Omega$, and $\E_h$ be the collection of all interior element edges.
The lowest order Raviart--Thomas finite element space~\cite{RT:1977} on $\Omega$ with the homogeneous essential boundary condition is given by 
\begin{equation*}
Y_h = \left\{ \p \in H_0 (\div ; \Omega) : \p|_{T} \in \mathcal{RT}_0 (T) \quad \forall T \in \T_h \right\},
\end{equation*}
where $\mathcal{RT}_0 (T)$ is the collection of all vector fields $\p$:~$T \rightarrow \mathbb{R}^2$ of the form
\begin{equation}
\label{RT}
\p (x_1, x_2) = \begin{bmatrix} a_1 + b_1 x_1 \\ a_2 + b_2 x_2\end{bmatrix}.
\end{equation}
The degrees of freedom for $Y_h$ are given by the average values of the normal components over the element edges.
We denote the degree of freedom of $\p \in Y_h$ associated to an edge $e \in \E_h$ by $(\p)_e$, i.e.,
\begin{equation*}
(\p)_e = \frac{1}{|e|} \int_{e} \p \cdot \n_e \, ds,
\end{equation*}
where $\n_e$ is a unit normal to $e$.
We define the space $X_h$ by
\begin{equation*}
X_h = \left\{ u \in L^2 (\Omega) : u|_T \textrm{ is constant} \quad \forall T \in \T_h \right\}.
\end{equation*}
Then it is clear that $\div \p \in X_h$ for $\p \in Y_h$.
Obviously, the degrees of freedom for $X_h$ are the values on the elements; for $u \in X_h$, $T \in \T_h$, and $x_T \in T$, we write
\begin{equation*}
(u)_T = u(x_T).
\end{equation*}

Let $\Pi_h$:~$H_0 (\div ; \Omega ) \rightarrow Y_h$ be the nodal interpolation operator, i.e., it satisfies
\begin{equation*}
(\Pi_h \p)_e = \frac{1}{|e|} \int_e \p \cdot \n_e \,ds, \quad \p \in H_0 (\div ; \Omega), e \in \E_h.
\end{equation*}
Then the following estimate holds~\cite[Lemma~5.1]{Oh:2013}.

\begin{lemma}
\label{Lem:interpolation}
Let $\p \in Y_h$ and $\theta$ be any continuous, piecewise linear, scalar function supported in $S \subset \Omega$.
Then, there exists a constant $c > 0$ independent of $|\Omega |$ and $h$ such that
\begin{equation*}
\int_S \left[\div (\Pi_h (\theta \p)) \right]^2 \,dx \leq c \int_S \left[ \div (\theta \p ) \right]^2 \,dx .
\end{equation*} 
\end{lemma}

Spaces $X_h$ and $Y_h$ are equipped with the inner products
\begin{eqnarray*}
\left< u, v \right>_{X_h} &=& h^2 \sum_{T \in \T_h} (u)_T (v)_T, \quad u,v \in X_h, \\
\left< \p, \q \right>_{Y_h} &=& h^2 \sum_{e \in \E_h} (\p)_e (\q)_e, \hspace{0.6cm} \p, \q \in Y_h,
\end{eqnarray*}
and their induced norms $\| \cdot \|_{X_h}$ and $\| \cdot \|_{Y_h}$, respectively.
We have the following facts on the norms $\| \cdot \|_{X_h}$ and $\| \cdot \|_{Y_h}$.

\begin{lemma}
\label{Lem:norm}
The norm $\| \cdot \|_{X_h} $ agrees with the $L^2 (\Omega)$-norm in ${X_h}$, i.e.,
\begin{equation*}
\| u \|_{X_h} = \| u \|_{L^2 (\Omega)} \quad \forall u \in X_h.
\end{equation*}
In addition, the norm $\| \cdot \|_{Y_h}$ is equivalent to the $(L^2 (\Omega))^2$-norm in $Y_h$ in the sense that there exist constants $\underbar{c}$, $\bar{c} > 0$ independent of $|\Omega |$ and $h$ such that
\begin{equation*}
\underbar{c} \| \p \|_{(L^2 (\Omega))^2} \leq \| \p \|_{Y_h} \leq \bar{c} \| \p \|_{(L^2 (\Omega))^2} \quad \forall \p \in Y_h.
\end{equation*}
\end{lemma}
\begin{proof}
See~\cite[Remark~2.2]{LPP:2019}.
\end{proof}

In this setting, the following inverse inequality which is useful in a selection of parameters for solvers for total variation minimization~(e.g.~\cite{BT:2009}) holds:

\begin{proposition}
\label{Prop:inverse}
For any $\p \in Y_h$, we have
\begin{equation}
\| \div \p \|_{X_h}^2 \leq \frac{8}{h^2} \| \p \|_{Y_h}^2 .
\end{equation}
\end{proposition}
\begin{proof}
See~\cite[Proposition~2.5]{LPP:2019}.
\end{proof}

In the rest of the paper, we may omit the subscripts $X_h$ and $Y_h$ from $\| \cdot \|_{X_h}$ and $\| \cdot \|_{Y_h}$, respectively, if there is no ambiguity.

We define the subset $C$ of $Y_h$ as
\begin{equation*}
C = \left\{ \p \in Y_h : |(\p)_e| \leq 1 \quad \forall e \in \E_h \right\}.
\end{equation*}
Then, we have the following discrete version of~\eqref{dual_model}:
\begin{equation}
\label{d_dual_model}
\min_{\p \in C} \left\{ \F(\p) := F^*( \div \p) \right\} .
\end{equation}
One may refer~\cite{HHSVW:2019,LP:2019a} for further details on~\eqref{d_dual_model}.
In the sequel, we denote a solution of~\eqref{d_dual_model} by $\p^* \in Y_h$.
In order to design a convergent additive Schwarz method for~\eqref{d_dual_model}, we require the following assumption on $F$, which is common in literature on Schwarz methods; see, e.g.,~\cite{BTW:2003,Tai:2003}.

\begin{assumption}
\label{Ass:regular}
{\rm The function $F$ in~\eqref{d_dual_model} is $\alpha$-strongly convex for some $\alpha > 0$, i.e., the map
\begin{equation*}
u \mapsto F(u) - \frac{\alpha}{2} \| u \|^2
\end{equation*}
is convex.
In addition, $F$ is Frech\'{e}t differentiable and its derivative $F'$ is $\beta$-Lipschitz continuous for some $\beta > 0$, i.e.,
\begin{equation*}
\| F'(u) - F'(v) \| \leq \beta \| u - v \| \quad \forall u, v \in X_h.
\end{equation*}}
\end{assumption}

We define the \textrm{condition number} $\kappa$ of $F$ by
\begin{equation}
\label{kappa}
\kappa = \frac{\beta}{\alpha},
\end{equation}
where $\alpha$ and $\beta$ were given in Assumption~\ref{Ass:regular}.
In the case when $F(u) = \frac{1}{2} \left<u, Ku \right> - \left<g, u \right>$ for some symmetric positive definite operator $K$:~$X_h \rightarrow X_h$ and $g \in X_h$, it is straightforward to see that $\kappa$ agrees with the ratio of the extremal eigenvalues of $K$. 

There are several interesting examples of $F$ satisfying Assumption~\ref{Ass:regular}.
First, we consider the ROF model~\eqref{ROF}, i.e., $F(u) = \frac{\lambda}{2} \|u - f\|^2$ for $\lambda > 0$ and $f \in X_h$.
Recall that the discrete ROF model defined on $X_h$ is given by
\begin{equation}
\label{ROF_App}
\min_{u \in X_h} \left\{ \frac{\lambda}{2} \| u - f \|^2 + TV_{\Omega} (u) \right\} .
\end{equation}
It is clear that Assumption~\ref{Ass:regular} is satisfied with $\alpha = \beta = \lambda$ and the condition number $\kappa = \beta / \alpha = 1$.
The second example is the $TV$-$H^{-1}$ model: a natural discretization of~\eqref{K} is given by
\begin{equation}
\label{K_App}
\min_{u \in X_h} \left\{ \frac{\lambda}{2} \| u - f \|_{K^{-1}}^2 + TV_{\Omega} (u)\right\},
\end{equation}
where $K$:~$X_h \rightarrow X_h$ is the standard 5-point-stencil approximation of $- \Delta$ with the homogeneous essential boundary condition~\cite{LeVeque:2007} and $\| v \|_{K^{-1}} = \left< K^{-1}v, v \right>^{1/2}$ for $v \in X_h$.
Since $K$ is nonsingular,~\eqref{K_App} satisfies Assumption~\ref{Ass:regular}.
It is well-known that the condition number of $K$ becomes larger as the image size grows; a detailed estimate for the condition number can be found in~\cite{LeVeque:2007}.
We note that a domain decomposition method for the $TV$-$H^{-1}$ model was previously considered in~\cite{Schonlieb:2009}.

If $F$ satisfies Assumption~\ref{Ass:regular}, we have the following properties on $F^*$~\cite{CP:2016}.

\begin{proposition}
\label{Prop:regular}
Under Assumption~\ref{Ass:regular}, the function $F^*$ in~\eqref{d_dual_model} is $(1/\beta$)-strongly convex and Frech\'{e}t differentiable with $(1/\alpha)$-Lipschitz continuous derivative.
Equivalently, the followings hold:
\begin{align*}
F^*(u) - F^*(v) - \left< (F^*)' (v), u-v \right> &\geq \frac{1}{2\beta} \| u - v \|^2 \quad \forall u, v \in X_h, \\
\| (F^*)'(u) - (F^*)'(v) \| &\leq \frac{1}{\alpha} \| u - v \| \quad\quad \forall u,v \in X_h.
\end{align*}
\end{proposition}

A solution of a discrete primal problem
\begin{equation*}
\min_{u \in X_h} \left\{ F(u) + TV_{\Omega}(u) \right\}
\end{equation*}
can be obtained from the solution $\p^*$ of~\eqref{d_dual_model} by solving
\begin{equation}
\label{pd_equiv}
\min_{u \in X_h} \left\{ F(u) - \left< u, \div \p^* \right> \right\}.
\end{equation}
See~\cite{CP:2016} for details.
Under Assumption~\ref{Ass:regular}, the problem~\eqref{pd_equiv} is smooth and strongly convex.
Therefore, one can solve~\eqref{pd_equiv} efficiently by linearly convergent first order methods such as~\cite[Algorithm~5]{CP:2016}.

The Bregman distance~\cite{Bregman:1967} associated with $\F$ is denoted by $D_{\F}$, i.e.,
\begin{equation}
\label{Bregman}
D_{\F}(\p, \q) = \F (\p) - \F (\q) - \left< \F' (\q), \p - \q \right>, \quad \p, \q \in Y_h.
\end{equation}
Note that
\begin{equation*}
\F' (\p) = \div^* \left( (F^*)' (\div \p) \right) , \quad \p \in Y_h,
\end{equation*}
where $\div^*$:~$X_h \rightarrow Y_h$ is the adjoint of $\div$:~$Y_h \rightarrow X_h$.
We have the following useful property of $D_{\F}$~\cite[Lemma~3.1]{CT:1993}.

\begin{lemma}
\label{Lem:three}
For any $\p$, $\q$, $\r \in Y_h$, we have
\begin{equation*}
D_{\F} (\r, \p) - D_{\F} (\r, \q) = D_{\F} (\q, \p) - \left< \F' (\p) - \F' (\q), \r - \q \right>.
\end{equation*}
\end{lemma}

For later use, we state the following trivial lemma for the set $C$.

\begin{lemma}
\label{Lem:L2_C}
For any $\p$, $\q \in C$, we have
\begin{equation*}
\| \p - \q \|^2 \leq 8 |\Omega|.
\end{equation*}
\end{lemma}
\begin{proof}
Note that $|\Omega| = m_1 m_2 h^2$ and $|\E_h| = (m_1 - 1)m_2 + m_1 (m_2 -1)$ for an image of the size $m_1 \times m_2$.
Using the inequality
\begin{equation*}
\| \p - \q \|^2 = h^2 \sum_{e \in \E_h} [ (\p)_e - (\q)_e ]^2 \leq 4h^2 |\E_h|,
\end{equation*}
the conclusion is easily acquired.
\end{proof}

Next, we present a space decomposition setting for $W = Y_h$.
Let $W_k$, $k = 1, \dots, N$ be subspaces of $W$ such that
\begin{equation}
\label{decomposition}
W = \sumk R_k^* W_k,
\end{equation}
where $R_k$:~$W \rightarrow W_k$ is the restriction operator and its adjoint $R_k^*$:~$W_k \rightarrow W$ is the natural extension operator.
We state an additional assumption on~\eqref{decomposition} inspired by~\cite{BTW:2003}.

\begin{assumption}
\label{Ass:stable}
{\rm There exist constants $c_1 > 0$ and $c_2 \geq 0$ which satisfy the following:
for any $\p$, $\q \in C$, there exists $\r_k \in W_k$~($k= 1, \dots, N$) such that
\begin{subequations}
\begin{equation}
\label{stable1}
\p - \q = \sumk R_k^* \r_k,
\end{equation}
\begin{equation}
\label{stable2}
\q + R_k^* \r_k \in C, \quad k = 1, \dots N,
\end{equation}
\begin{equation}
\label{stable3}
\sumk \| \div R_k^* \r_k \|^2 \leq c_1 \| \div (\p - \q) \|^2 + c_2 \| \p - \q \|^2 .
\end{equation}
\end{subequations}}
\end{assumption}

In Assumption~\ref{Ass:stable}, we call $\{ \r_k \}$ a \textit{stable decomposition} of $\p - \q$.
A particular choice of spaces $\{ W_k \}$ and functions $\{ \r_k \}$ satisfying Assumption~\ref{Ass:stable} based on an overlapping domain decomposition of $\Omega$ will be given in Section~\ref{Sec:DD}.

We conclude this section by presenting two useful estimates for sequences of positive real numbers.

\begin{lemma}
\label{Lem:recur1}
Let $\{a_n \}_{n \geq 0}$ be a sequence of positive real numbers which satisfies
\begin{equation*}
a_n - a_{n+1} \geq \frac{1}{c^2} (a_{n+1} - \gamma a_n )^2, \quad n \geq 0
\end{equation*}
for some $c > 0$ and $0 \leq \gamma < 1$.
Then we have
\begin{equation*}
a_n \leq \frac{1}{\tc n + 1/a_0},
\end{equation*}
where
\begin{equation*}
\tc = \frac{(1- \gamma)^2}{2a_0 (1-\gamma)^2 + (\gamma \sqrt{a_0} + c)^2}.
\end{equation*}
\end{lemma}
\begin{proof}
See~\cite[Lemma~3.5]{CTWY:2015}.
\end{proof}

\begin{lemma}
\label{Lem:recur2}
Let $\{a_n \}_{n \geq 0}$ be a sequence of positive real numbers which satisfies
\begin{equation*}
a_{n+1} \leq \gamma a_n + c, \quad n \geq 0
\end{equation*}
for some $c > 0$ and $0 \leq \gamma < 1$.
Then we have
\begin{equation*}
a_n \leq \gamma^n \left( a_0 - \frac{c}{1-\gamma} \right) + \frac{c}{1-\gamma}.
\end{equation*}
\end{lemma}
\begin{proof}
It is elementary.
\end{proof}

\begin{remark}
\label{Rem:recurrence}
In~\cite[Lemma~3.5]{CTWY:2015}, it was proved that the sequence in Lemma~\ref{Lem:recur1} satisfies
\begin{equation}
\label{recurrence1}
a_n - a_{n+1} \geq \left( \frac{1-\gamma}{\gamma \sqrt{a_0} + c} \right)^2 a_{n+1}^2.
\end{equation}
Then~\eqref{recurrence1} was combined with~\cite[Lemma~3.2]{TX:2002} to yield the desired result.
We note that several alternative estimates for the $O(1/n)$ convergence of~\eqref{recurrence1} with different constants are available; see~\cite[Lemmas~3.6 and~3.8]{Beck:2015} and~\cite[Proposition~3.1]{CP:2015}.
\end{remark}

\section{An additive Schwarz method}
\label{Sec:Schwarz}
In this section, we propose an abstract additive Schwarz method for dual total variation minimization~\eqref{d_dual_model} in terms of the space decomposition~\eqref{decomposition}.
Then, several remarkable convergence properties of the proposed method are investigated.
Algorithm~\ref{Alg:ASM} shows the proposed additive Schwarz method for~\eqref{d_dual_model}.

\begin{algorithm}[]
\caption{Additive Schwarz method for dual total variation minimization~\eqref{d_dual_model}}
\begin{algorithmic}[]
\label{Alg:ASM}
\STATE Choose $\p^{(0)} \in C$ and $\tau \in (0, 1/N]$.
\FOR{$n=0,1,2,\dots$}
\item \begin{eqnarray*}
\r_k^{(n+1)} &\in& \argmin_{\r_k \in W_k,  \p^{(n)} + R_k^* \r_k \in C } \F \left(\p^{(n)} + R_k^* \r_k \right), \quad k=1,\dots, N \\
\p^{(n+1)} &=& \p^{(n)} + \tau \sumk R_k^* \r_k^{(n+1)}
\end{eqnarray*}
\ENDFOR
\end{algorithmic}
\end{algorithm}

We note that an additive Schwarz method for the dual ROF model based on a constraint decomposition of $C$ was proposed in~\cite{CTWY:2015}, which is slightly different from our method.
Differently from that in~\cite{CTWY:2015}, Algorithm~\ref{Alg:ASM} does not require an explicit decomposition of the constraint set $C$ as in~\cite[Proposition~2.1]{CTWY:2015}.
A similar situation was previously addressed in~\cite{BTW:2003,Tai:2003} for the obstacle problem.
In addition, the proposed method is applicable to not only the ROF model but also general total variation minimization problems satisfying Assumption~\ref{Ass:regular}.

In order to analyze the convergence of Algorithm~\ref{Alg:ASM}, we first present a descent rule for Algorithm~\ref{Alg:ASM} in Lemma~\ref{Lem:descent}, which is a main tool for the convergence analysis.
We note that arguments using the descent rule are standard in the analysis of first-order methods for convex optimization~(see~\cite[Lemma~2.3]{BT:2009},~\cite[equation~(3.6)]{CP:2015}, and~\cite[equations~(4.37) and (C.1)]{CP:2016} for instance).
The proof of Lemma~\ref{Lem:descent} closely follows that of~\cite[Lemma~3.2]{LP:2019b}.
The main difference is that Lemma~\ref{Lem:descent} allows any smooth $\F$ using the notion of Bregman distance, while~\cite{LP:2019b} is restricted to the dual ROF case. In addition, the decomposition of $\p - \q$ is not unique due to the overlapping of subdomains, while it is unique in the nonoverlapping case given in~\cite{LP:2019b}.

\begin{lemma}
\label{Lem:descent}
Let $\p$, $\q \in C$ and $\tau \in (0, 1/N]$.
We define $\hr_k \in W_k$, $\hq \in Y$, and $\br_k \in W_k$, $k=1, \dots , N$, as follows.
\begin{enumerate}[label=\emph{(\roman*)}]
\item $\displaystyle \hr_k \in \argmin_{\r_k \in W_k, \q + R_k^* \r_k \in C} \F \left( \q + R_k^* \r_k \right)$, $k=1, \dots, N$.
\item $\displaystyle \hq = \q + \tau \sumk R_k^* \hr_k$.
\item $\{ \br_k \}$: a stable decomposition of $\p - \q$ in Assumption~\ref{Ass:stable}.
\end{enumerate}
Then we have
\begin{equation*}
\tau \F(\p) + (1- \tau) \F(\q) - \F(\hq) \geq \tau \sumk \left( D_{\F}(\q + R_k^* \br_k , \q + R_k^* \hr_k ) - D_{\F} (\q + R_k^* \br_k ,\q) \right).
\end{equation*}
\end{lemma}
\begin{proof}
As $\q + R_k^* \br_k \in C$ by~\eqref{stable2}, from the optimality condition of $\hr_k$~(cf.~\cite[Lemma~2.2]{LP:2019b}), we have
\begin{equation*}
\F(\q + R_k^* \br_k ) - \F(\q + R_k^* \hr_k ) \geq D_{\F} (\q + R_k^* \br_k , \q + R_k^* \hr_k).
\end{equation*}
Summation of the above equation over $k=1, \dots, N$ yields
\begin{equation}
\label{descent1}
\tau \sumk \left( \F(\q + R_k^* \br_k) - \F(\q + R_k^* \hr_k )\right) \geq \tau \sumk D_{\F} (\q + R_k^* \br_k , \q + R_k^* \hr_k). 
\end{equation}
Note that
\begin{equation*}
\hq = \q + \tau \sumk R_k^* \hr_k = (1 - \tau N) \q + \tau \sumk (\q + R_k^* \hr_k ).
\end{equation*}
Since $1 - \tau N \geq 0$, we obtain by the convexity of $\F$ that
\begin{equation}
\label{descent2}
(1 - \tau N ) \F(\q) + \tau \sumk \F(\q + R_k^* \hr_k) \geq \F(\hq).
\end{equation}
On the other hand, by the definition of Bregman distance~\eqref{Bregman}, the condition~(iii), and the convexity of $\F$, we have
\begin{equation} \begin{split}
\label{descent3}
\tau \left( N \F(\q) - \sumk \F(\q + R_k^* \br_k) \right)
&= -\tau \sumk \left( \left< \F '(\q), R_k^* \br_k \right> + D_{\F} (\q + R_k^* \br_k, \q) \right) \\
&= -\tau \left< \F ' (\q), \p - \q \right> -  \tau \sumk D_{\F} (\q + R_k^* \br_k, \q) \\
&\geq -\tau (\F (\p) - \F (\q)) -  \tau \sumk D_{\F} (\q + R_k^* \br_k, \q).
\end{split} \end{equation}
Summation of~\eqref{descent1},~\eqref{descent2}, and~\eqref{descent3} completes the proof.
\end{proof}

With Lemma~\ref{Lem:descent}, the following property of sufficient decrease for Algorithm~\ref{Alg:ASM} is straightforward~(cf.~\cite[Lemma~3.3]{CTWY:2015}).

\begin{lemma}
\label{Lem:decrease}
In Algorithm~\ref{Alg:ASM}, we have
\begin{equation*}
\F(\p^{(n)}) - \F(\p^{(n+1)}) \geq \frac{\tau}{2\beta} \sumk \| \div R_k^* \r_k^{(n+1)} \|^2 , \quad n \geq 0,
\end{equation*}
where $\beta$ was given in Assumption~\ref{Ass:regular}.
\end{lemma}
\begin{proof}
Substitute $\p$ by $\p^{(n)}$, $\q$ by $\p^{(n)}$, and $\br_k$ by $0$ in Lemma~\ref{Lem:descent}.
Then $\hr_k = \r_k^{(n+1)}$, $\hq = \p^{(n+1)}$, and we obtain
\begin{equation*}
\F(\p^{(n)}) - \F(\p^{(n+1)}) \geq \tau \sumk D_{\F} (\p^{(n)} , \p^{(n)} + R_k^* \r_k^{(n+1)}).
\end{equation*}
On the other hand, for any $k$, it follows by Proposition~\ref{Prop:regular} that
\begin{equation*} \begin{split}
D_{\F} &(\p^{(n)}, \p^{(n)} + R_k^* \r_k^{(n+1)}) \\
&= \F(\p^{(n)}) - \F(\p^{(n)} + R_k^* \r_k^{(n+1)}) - \left< \F' (\p^{(n)} + R_k^* \r_k^{(n+1)}), -R_k^* \r_k^{(n+1)} \right> \\
&= F^*(\div \p^{(n)}) - F^*(\div (\p^{(n)} + R_k^* \r_k^{(n+1)})) \\
&\quad  - \left< (F^*)' (\div (\p^{(n)} + R_k^* \r_k^{(n+1)}) ), - \div R_k^* \r_k^{(n+1)} \right> \\
&\geq \frac{1}{2\beta} \| \div R_k^* \r_k^{(n+1)} \|^2.
\end{split} \end{equation*}
Thus, we readily get the desired result.
\end{proof}

Using Lemmas~\ref{Lem:descent} and~\ref{Lem:decrease}, one can show the $O(1/n)$ convergence of Algorithm~\ref{Alg:ASM} as follows.

\begin{theorem}
\label{Thm:rate}
In Algorithm~\ref{Alg:ASM}, let $\zeta_n = \alpha ( \F(\p^{(n)}) - \F(\p^*) )$ for $n \geq 0$.
Then there exists a constant $\tc > 0$ depending only on $|\Omega|$, $\kappa$, $\tau$, $\zeta_0$, $c_1$, and $c_2$ such that
\begin{equation*}
\zeta_n \leq \frac{1}{\tc n + 1/\zeta_0},
\end{equation*}
where $\kappa$ was given in~\eqref{kappa} and $c_1$, $c_2$ were given in Assumption~\ref{Ass:stable}.
\end{theorem}
\begin{proof}
In Lemma~\ref{Lem:descent}, set $\p$ by $\p^*$, $\q$ by $\p^{(n)}$, and $\hr_s$ by $\r_s^{(n+1)}$.
Then $\hq = \p^{(n+1)}$.
We obtain
\begin{equation*} \begin{split}
&\quad \tau \F(\p^*) + (1-\tau) \F(\p^{(n)}) - \F(\p^{(n+1)}) \\
&\geq \tau \sumk \left( D_{\F} (\p^{(n)} + R_k^* \br_k , \p^{(n)} + R_k^* \r_k^{(n+1)}) - D_{\F} (\p^{(n)} + R_k^* \br_k , \p^{(n)}) \right)\\
&= \tau \sumk \left( D_{\F} (\p^{(n)} , \p^{(n)} + R_k^* \r_k^{(n+1)}) - \left< \F ' (\p^{(n)} + R_k^* \r_k^{(n+1)}) - \F ' (\p^{(n)}) , R_k^* \br_k \right> \right),
\end{split} \end{equation*}
where the last equality is due to Lemma~\ref{Lem:three}.
It follows by Proposition~\ref{Prop:regular} and the Cauchy--Schwarz inequality that
\begin{equation} \begin{split}
\label{rate1}
\zeta_{n+1} &- (1-\tau) \zeta_n \leq - \tau \alpha \sumk \Big( D_{\F} (\p^{(n)} , \p^{(n)} + R_k^* \r_k^{(n+1)}) \\
&\hspace{3cm} -  \left< \F ' (\p^{(n)} + R_k^* \r_k^{(n+1)}) - \F ' (\p^{(n)}) , R_k^* \br_k \right> \Big) \\
&\leq \tau \alpha \sumk \left< \F ' (\p^{(n)} + R_k^* \r_k^{(n+1)}) - \F ' (\p^{(n)}) , R_k^* \br_k \right> \\
&= \tau \alpha \sumk \left< (F^*)' (\div (\p^{(n)} + R_k^* \r_k^{(n+1)})) - (F^*)' (\div \p^{(n)}), \div R_k^* \br_k \right> \\
&\leq \tau \sumk \| \div R_k^* \r_k^{(n+1)} \| \| \div R_k^* \br_k \| \\
&\leq \tau  \left( \sumk \| \div R_k^* \r_k^{(n+1)} \|^2 \right)^{\frac{1}{2}} \left( \sumk \| \div R_k^* \br_k \|^2 \right)^{\frac{1}{2}} .
\end{split} \end{equation}
By~\eqref{stable3}, Proposition~\ref{Prop:regular}, and Lemma~\ref{Lem:L2_C}, we have
\begin{equation} \begin{split}
\label{rate2}
\sumk \| \div R_k^* \br_k \|^2 &\leq c_1 \| \div (\p^{(n)} - \p^* ) \|^2 + c_2 \| \p^{(n)} - \p^* \|^2 \\
&\leq 2\kappa c_1 \zeta_n + 8c_2 |\Omega| \\
&\leq 2\kappa c_1 \zeta_0 + 8c_2 |\Omega| =: c_3.
\end{split} \end{equation}
Thus, it is satisfied by~\eqref{rate1} and~\eqref{rate2} that
\begin{equation}
\label{rate4}
\zeta_{n+1} - (1-\tau) \zeta_n \leq \tau c_3^{\frac{1}{2}} \left( \sumk \| \div R_k^* \r_k^{(n+1)} \|^2 \right)^{\frac{1}{2}}.
\end{equation}
Combining~\eqref{rate4} with Lemma~\ref{Lem:decrease}, we get
\begin{equation}
\label{rate3}
\zeta_n - \zeta_{n+1} \geq \frac{1}{2\tau \kappa c_3} \left[\zeta_{n+1} - (1-\tau) \zeta_n \right]^2.
\end{equation}
Finally, invoking Lemma~\ref{Lem:recur1} to~\eqref{rate3} completes the proof.
\end{proof}

To see the dependency of the convergence of Algorithm~\ref{Alg:ASM} on parameters, we do some additional calculations starting from~\eqref{rate3}.
By Lemma~\ref{Lem:recur1}, we have
\begin{equation*}
\gamma = 1-\tau \quad\textrm{and}\quad c = \sqrt{2\tau \kappa c_3}= 2 \sqrt{\tau \kappa (\kappa c_1 \zeta_0 + 4c_2 |\Omega|)},
\end{equation*}
so that
\begin{equation} \begin{split}
\label{const}
\frac{1}{\tilde{c}} &= 2\zeta_0 + \left(\frac{(1-\tau) \sqrt{\zeta_0} + 2\sqrt{\tau \kappa (\kappa c_1 \zeta_0 + 4c_2 |\Omega|)} }{\tau} \right)^2 \\
&\leq 2\left[ 1 + \left(\frac{1-\tau}{\tau}\right)^2 + \frac{4 \kappa^2 c_1}{\tau } \right]\zeta_0 + \frac{32 \kappa c_2 |\Omega|}{\tau}.
\end{split} \end{equation}
Hence, the constant $\tilde{c}$ in Theorem~\ref{Thm:rate} depends on $|\Omega|$, $\kappa$, $\tau$, $\zeta_0$, $c_1$, and $c_2$ only.
Moreover, we observe that~\eqref{const} is decreasing with respect to $\tau \in (0, 1/N]$.
Hence, we may choose $\tau = 1/N$; see also~\cite[Remark~3.1]{CTWY:2015}.

\begin{remark}
\label{Rem:CTWY}
In the ROF case~\eqref{ROF},  alternatively to Theorem~\ref{Thm:rate}, one can prove the $O(1/n)$ convergence rate of Algorithm~\ref{Alg:ASM} by a similar argument as in~\cite{CTWY:2015}.
Compared to~\cite{CTWY:2015}, our proof is simple due to the descent rule, Lemma~\ref{Lem:descent}.
Moreover, our estimate is independent of $N$ while that in~\cite{CTWY:2015} is not.
\end{remark}

In addition to the $O(1/n)$ convergence, we prove that Algorithm~\ref{Alg:ASM} converges pseudo-linearly, i.e., $\F(\p^{(n)})$ decreases as fast as linear convergence until it reaches a particular value.
Theorem~\ref{Thm:pseudo} provides a rigorous statement for the pseudo-linear convergence of Algorithm~\ref{Alg:ASM}.

\begin{theorem}
\label{Thm:pseudo}
In Algorithm~\ref{Alg:ASM}, let $\zeta_n = \alpha ( \F(\p^{(n)}) - \F(\p^*) )$ for $n \geq 0$.
Then we have
\begin{equation*}
\zeta_n \leq \left(1 - \frac{\tau }{\kappa^2 (\sqrt{c_1} + \sqrt{c_1 + \kappa^{-2}})^2} \right)^n \left( \zeta_0 - \frac{4c_2 |\Omega|}{\kappa \sqrt{c_1(c_1+ \kappa^{-2})}} \right) + \frac{4c_2 |\Omega|}{\kappa \sqrt{c_1(c_1+ \kappa^{-2})}},
\end{equation*}
where $\kappa$ was given in~\eqref{kappa} and $c_1$, $c_2$ were given in Assumption~\ref{Ass:stable}.
\end{theorem}
\begin{proof}
Take any $n \geq 0$.
For the sake of convenience, we write
\begin{equation*}
\Delta = \frac{1}{2} \sumk \| \div R_k^* \r_k^{(n+1)} \|^2.
\end{equation*}
The starting points of the proof are~\eqref{rate1} and~\eqref{rate2}:
\begin{equation*}
\zeta_{n+1} \leq (1 - \tau) \zeta_n + \tau (2\kappa c_1 \zeta_n + 8 c_2 |\Omega|)^{\frac{1}{2}}(2\Delta)^{\frac{1}{2}}.
\end{equation*}
Using the inequality
\begin{equation*}
ab \leq \epsilon a^2 + \frac{1}{4\epsilon}b^2, \quad 0 < \epsilon < 1,
\end{equation*}
we readily get
\begin{equation} \begin{split}
\label{pseudo1}
\zeta_{n+1} &\leq (1 - \tau) \zeta_n + \tau \left( \zeta_n + \frac{4c_2 |\Omega|}{\kappa c_1} \right)^{\frac{1}{2}} \left( 4\kappa c_1 \Delta \right)^{\frac{1}{2}} \\
&\leq (1 - \tau) \zeta_n + \tau \left[\epsilon \left( \zeta_n + \frac{4c_2 |\Omega|}{\kappa c_1} \right) + \frac{1}{4\epsilon} \cdot 4\kappa c_1 \Delta \right] \\
&= (1-\tau + \tau \epsilon) \zeta_n + \frac{\tau  \kappa c_1}{\epsilon} \Delta +  \frac{4 \tau \epsilon c_2 |\Omega|}{\kappa c_1} \\
&\leq (1-\tau + \tau \epsilon) \zeta_n + \frac{ \kappa^2 c_1}{\epsilon} (\zeta_n - \zeta_{n+1}) +  \frac{4 \tau \epsilon c_2 |\Omega|}{\kappa c_1},
\end{split} \end{equation}
where the last inequality is due to Lemma~\ref{Lem:decrease}.
The equation~\eqref{pseudo1} can be rewritten as
\begin{equation*}
\zeta_{n+1} \leq \left( 1 - \frac{\tau \epsilon ( 1 - \epsilon )}{ \epsilon + \kappa^2 c_1} \right) \zeta_n + \frac{4\tau \epsilon^2 c_2 |\Omega|}{\kappa c_1 ( \epsilon + \kappa^2 c_1 )}.
\end{equation*}
We take
\begin{equation*}
\epsilon = \kappa^2 \left( \sqrt{ c_1 \left( c_1 + \kappa^{-2} \right)} - c_1 \right) \in (0, 1)
\end{equation*}
which maximizes $\frac{\epsilon ( 1 - \epsilon )}{ \epsilon + \kappa^2 c_1}$ so that
\begin{align*}
\frac{ \epsilon ( 1 - \epsilon )}{ \epsilon + \kappa^2 c_1} &= \frac{1}{ \kappa^2 (\sqrt{c_1} + \sqrt{c_1 + \kappa^{-2}})^2 }, \\
\frac{ \epsilon^2}{ c_1 (\epsilon + \kappa^2 c_1 ) } &= \frac{1}{\kappa^2 \sqrt{c_1 ( c_1 + \kappa^{-2})} (\sqrt{c_1} + \sqrt{c_1 + \kappa^{-2} })^2 }.
\end{align*}
We note that similar computations were done in~\cite{BTW:2003,Tai:2003}.
Then it follows that
\begin{equation*}
\zeta_{n+1} \leq \left( 1 - \frac{\tau}{\kappa^2 (\sqrt{c_1} + \sqrt{c_1 + \kappa^{-2}})^2 } \right) \zeta_n + \frac{4\tau c_2 |\Omega|}{\kappa^3 \sqrt{c_1 ( c_1 + \kappa^{-2})} (\sqrt{c_1} + \sqrt{c_1 + \kappa^{-2}})^2 }.
\end{equation*}
By Lemma~\ref{Lem:recur2}, we have
\begin{equation*}
\zeta_n \leq \left(1 - \frac{\tau }{\kappa^2 (\sqrt{c_1} + \sqrt{c_1 + \kappa^{-2}})^2} \right)^n \left( \zeta_0 - \frac{4c_2 |\Omega|}{\kappa \sqrt{c_1(c_1+ \kappa^{-2})}} \right) + \frac{4c_2 |\Omega|}{\kappa \sqrt{c_1(c_1+ \kappa^{-2})}},
\end{equation*}
which is the desired result.
\end{proof}

By Theorem~\ref{Thm:pseudo}, $\{ \F(\p^{(n)}) \}$ in Algorithm~\ref{Alg:ASM} converges pseudo-linearly with threshold $4c_2 |\Omega|/\kappa \sqrt{c_1 (c_1 +\kappa^{-2} )}$.
It means that if one can make $c_2 |\Omega| / c_1$ sufficiently small, then the proposed method shows almost the same convergence pattern as a linearly convergent algorithm.
We shall consider in Section~\ref{Sec:DD} that how to make $c_2 |\Omega| / c_1$ small, and observe the behavior of the proposed method in Section~\ref{Sec:Applications}.

\section{Domain decomposition}
\label{Sec:DD}
In this section, we present overlapping domain decomposition settings for the proposed DDM.
We decompose the domain $\Omega$ into $\N$ disjoint square subdomains $\{ \Omega_s \}_{s=1}^{\N}$ in a checkerboard fashion.
The side length of each subdomain $\Omega_s$ is denoted by $H$.
For each $s=1, \dots, \N$, let $\Omega_s'$ be an enlarged subdomain consisting of $\Omega_s$ and its surrounding layers of pixels with width $\delta$ for some $\delta > 0$.
Overlapping subdomains $\{ \Omega_s' \}_{s=1}^{\N}$ can be colored with $N_c \leq 4$ colors such that any two subdomains can be of the same color if they are disjoint~\cite{CTWY:2015}.
Let $S_k$ be the union of all subdomains $\Omega_s'$ with color $k$ for $k = 1, \dots, N_c$.
We denote the collection of all elements of $\T_h$ in $S_k$ by $\T_{h,k}$.
In what follows, for two positive real numbers $A$ and $B$ depending on the parameters $|\Omega|$, $H$, $h$, and $\delta$, we use the notation $A \lesssim B$ to represent that there exists a constant $c > 0$ such that $A \leq c B$, where $c$ is independent of the parameters.
In addition, we write $A \approx B$ if $A \lesssim B$ and $B \lesssim A$.

We consider a DDM based on the domain decomposition $\{ \Omega_s' \}$.
Let $N = N_c$ and for $k = 1, \dots, N$, we set $W_k = Y_h (S_k )$, where
\begin{equation}
\label{1L}
Y_h (S_k ) = \left\{ \p \in H_0 (\div; S_k ) : \p |_T \in \mathcal{RT}_0 (T) \quad \forall T \in \T_{h,k} \right\},
\end{equation}
where $\mathcal{RT}_0 (T)$ was defined in~\eqref{RT}.

By Lemma~3.4 in~\cite{TW:2005}, there exists a continuous and piecewise linear partition of unity $\{ \theta_k \}_{k=1}^{N}$ for $\Omega$ subordinate to the covering $\{ S_k \}_{k=1}^{N}$ such that
\begin{subequations}
\label{pou}
\begin{equation}
\supp \theta_k \subset \bar{S}_k, \quad \theta_k \in W_0^{1, \infty} (S_k),
\end{equation}
\begin{equation}
0 \leq \theta_k \leq 1, \quad \sum_{k=1}^{N} \theta_k = 1 \textrm{ in } \Omega,
\end{equation}
\begin{equation}
\| \nabla \theta_k \|_{L^{\infty}(S_k)} \lesssim \frac{1}{\delta},
\end{equation}
\end{subequations}
where $\bar{S}_k$ is the closure of $S_k$ and $W_0^{1, \infty}(S_k)$ is defined as
\begin{equation*}
W_0^{1, \infty}(S_k) = \left\{ \theta \in L^{\infty} (S_k ) : \nabla \theta \in L^{\infty} (S_k ) \textrm{ and } \theta|_{\partial S_k} = 0 \right\}.
\end{equation*}
One can show the following property of $\theta_k$.

\begin{proposition}
\label{Prop:pou}
For $1 \leq k \leq N$ and $\p \in Y_h$, we have
\begin{equation*}
\| \div (\Pi_h (\theta_k \p)) \|^2 \lesssim \| \div \p \|^2 + \frac{1}{\delta^2} \| \p \|^2.
\end{equation*}
\end{proposition}
\begin{proof}
Invoking~\eqref{pou}, Lemmas~\ref{Lem:interpolation} and~\ref{Lem:norm} yields
\begin{equation*} \begin{split}
\| \div (\Pi_h (\theta_k \p)) \|^2 &= \int_{\Omega} [ \div (\Pi_h (\theta_k \p)) ]^2 \,dx \\
& \lesssim \int_{\Omega} [\div (\theta_k \p )]^2 \,dx \\
&\lesssim \int_{\Omega} [\nabla \theta_k \cdot \p]^2 \,dx + \int_{\Omega} [\theta_k \div \p]^2 \,dx  \\
&\lesssim \frac{1}{\delta^2} \int_{\Omega} |\p|^2 \,dx + \int_{\Omega} (\div \p)^2 \,dx \\
&\lesssim \frac{1}{\delta^2} \| \p \|^2 + \| \div \p \|^2.
\end{split} \end{equation*}
We note that a similar calculation was done in~\cite[Lemma~3.2]{CTWY:2015}.
\end{proof}

Using Proposition~\ref{Prop:pou}, the following stable decomposition estimate is obtained.

\begin{lemma}
\label{Lem:1L}
In the space decomposition setting~\eqref{1L}, Assumption~\ref{Ass:stable} holds with
\begin{equation*}
c_1 \approx 1, \quad c_2 \lesssim \frac{1}{\delta^2}.
\end{equation*}
\end{lemma}
\begin{proof}
Clearly, we have $c_1 \geq 1$.
Take any $\p, \q \in C$.
For $k = 1, \dots, N$, we define $\r_k \in Y_k$ by
\begin{equation*}
R_k^* \r_k = \Pi_h (\theta_k (\p - \q)).
\end{equation*}
It is obvious that $\{ \r_k \}$ satisfies~\eqref{stable1} and~\eqref{stable2}.
By Proposition~\ref{Prop:pou}, we get
\begin{equation*}
\| \div R_k^* \r_k \|^2 \lesssim \| \div ( \p - \q ) \|^2 + \frac{1}{\delta^2} \| \p - \q \|^2 .
\end{equation*} 
Summing the above equation over all $k$ yields~\eqref{stable3} with $c_1 \lesssim 1$ and $c_2 \lesssim 1/\delta^2$.
\end{proof}

\begin{remark}
\label{Rem:nonoverlapping}
Lemma~\ref{Lem:1L} cannot be applied to the nonoverlapping case~($\delta = 0$) since $1/ \delta^2 \rightarrow \infty$ as $\delta \rightarrow 0$.
On the other hand, in a finite difference discretization given in~\cite{Chambolle:2004,LP:2019b}, it can be proved that the nonoverlapping decomposition satisfies Assumption~\ref{Ass:stable} with a similar argument to~\cite[Lemma~3.5]{LP:2019b}.
\end{remark}

Combining Theorem~\ref{Thm:pseudo} and Lemma~\ref{Lem:1L}, we get the following result.

\begin{corollary}
\label{Cor:1L}
For fixed $\tau > 0$, Algorithm~\ref{Alg:ASM} with the domain decomposition~\eqref{1L} converges pseudo-linearly at rate $\gamma$ with threshold $\epsilon > 0$ such that $\epsilon \lesssim |\Omega|/\delta^2$ and $\gamma$ is independent of $|\Omega|$, $H$, $h$, and $\delta$.
\end{corollary}

From Corollary~\ref{Cor:1L}, we can deduce several notable facts about Algorithm~\ref{Alg:ASM}.
As $\epsilon \lesssim |\Omega| / \delta^2$, the proposed DDM converges as fast as a linear convergent algorithm until the energy error becomes very small if $\delta$ is largely chosen.
Indeed, we will see in Section~\ref{Sec:Applications} that the energy error decreases linearly to the machine error if $\delta$ is chosen such that $|\Omega|^{1/2}/\delta$ is less than about $2^7$.
Moreover, since $\gamma$ does not depend on $|\Omega|$, $H$, $h$ or $\delta$, the linear convergence rate of Algorithm~\ref{Alg:ASM} which dominates the convergence behavior is the same regardless of $|\Omega |$, $\delta$, and the number of subdomains.
To the best of our knowledge, such an observation is new in the field of DDMs.
Usually, the linear convergence rate of additive Schwarz methods depends on $\delta$; see~\cite{Tai:2003} for example.
However, in our case, the value of $\delta$ affects only the threshold $\epsilon$ but not the rate $\gamma$.

In the DDM described above, local problems in $\Omega_s'$~($s = 1, \dots, \N )$ has the following general form:
\begin{equation}
\label{local1}
\min_{\r_s \in Y_h (\Omega_s '), \q + (R^s)^* \r_s \in C } \F \left( \q + (R^s)^*\r_s \right),
\end{equation}
where $\q \in Y_h$, $Y_h (\Omega_s ')$ is defined in the same manner as~\eqref{1L}, and $(R^s)^*$:~$Y_h (\Omega_s ') \rightarrow Y_h$ is the natural extension operator.
Let $\p_s = \r_s + R^s \q$. Then~\eqref{local1} is equivalent to
\begin{equation}
\label{local2}
\min_{\p_s \in C^s} \left\{ \F_s (\p_s ) := F^* \left( \div (R^s)^* \p_s + g_s \right) \right\},
\end{equation}
where $C^s$ is the subset of $Y_h (\Omega_s ')$ defined by
\begin{equation*}
C^s = \left\{ \p_s \in Y_h (\Omega_s ') : |(\p_s)_e| \leq 1 \quad \forall e \in \E_h \textrm{ such that $e$ is in the interior of $\Omega_s '$} \right\}
\end{equation*}
and
\begin{equation*}
g_s = \div (I - (R^s)^*R^s) \q.
\end{equation*}
Existing state-of-the-art solvers for~\eqref{d_dual_model}~(see~\cite{CP:2016}) can be utilized to solve~\eqref{local2}; we have
\begin{equation*}
\F_s '(\p_s ) = R^s \div^* \left( (F^*)' \left( \div (R^s)^* \p_s + g_s \right) \right).
\end{equation*}

\begin{remark}
\label{Rem:Helmholtz}
For unconstrained, strongly convex, vector-valued problems such as the $\mathrm{grad}$-$\div$ problem, one can obtain a stable decomposition such that $c_1$ is dependent on $\delta$ and $c_2 = 0$ by using the discrete Helmholtz decomposition~(see, e.g.,~\cite[Lemma~5.8]{Oh:2013}).
In this case, a linear convergence rate depending on $\delta$ is obtained by the same argument as Theorem~\ref{Thm:pseudo}.
However, it seems that such a stable decomposition is not available in our case: constrained and non-strongly convex problem; see Table~\ref{Table:difficult}.
Numerical experiments presented in Section~\ref{Sec:Applications} will show the following phenomena: Algorithm~\ref{Alg:ASM} converges not linearly but pseudo-linearly, i.e., the convergence rate deteriorates when $\F(\p^{(n)}) - \F(\p^* )$ becomes sufficiently small, and the linearly convergent part of Algorithm~\ref{Alg:ASM} is not dependent on $\delta$.
\end{remark}

\section{Applications}
\label{Sec:Applications}
\begin{figure}[]
\centering
\subfloat[][Peppers $512\times 512$]{ \includegraphics[height=3.8cm]{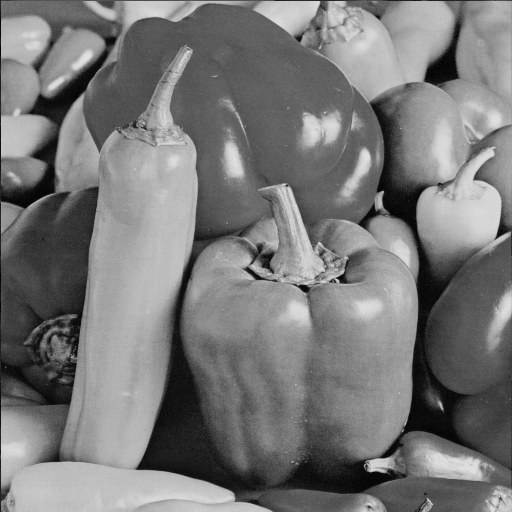} } \quad
\subfloat[][Noisy image~(PSNR: 19.11)]{ \includegraphics[height=3.8cm]{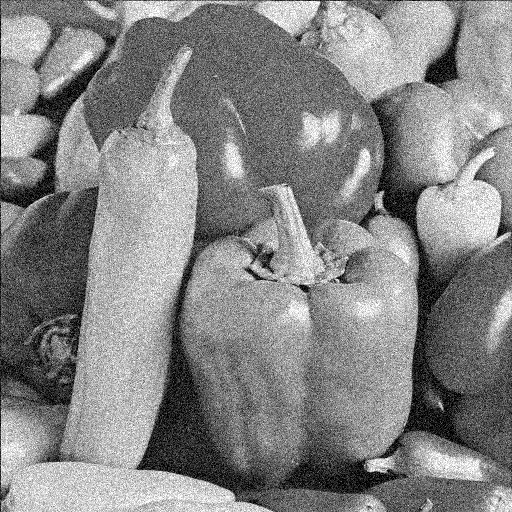} } \quad
\subfloat[][ROF, $\N = 16 \times 16$ (PSNR: 24.41)]{ \includegraphics[height=3.8cm]{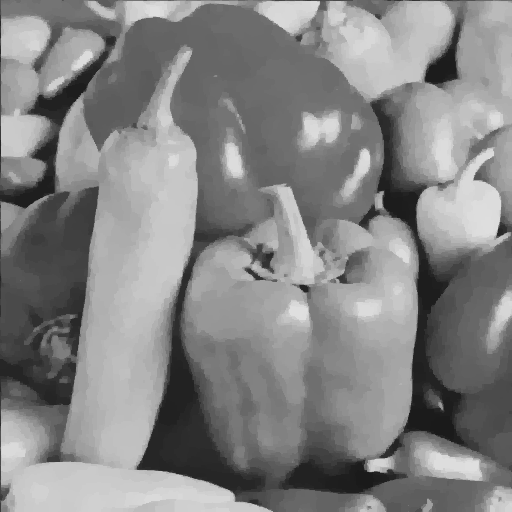} } \\
\subfloat[][Cameraman $2048 \times 2048$]{ \includegraphics[height=3.8cm]{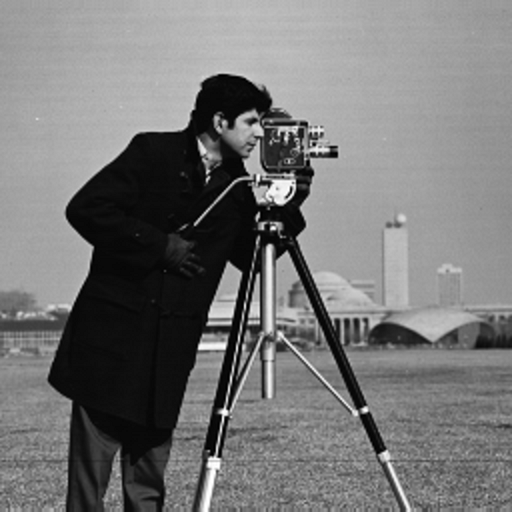} } \quad
\subfloat[][Noisy image~(PSNR: 19.17)]{ \includegraphics[height=3.8cm]{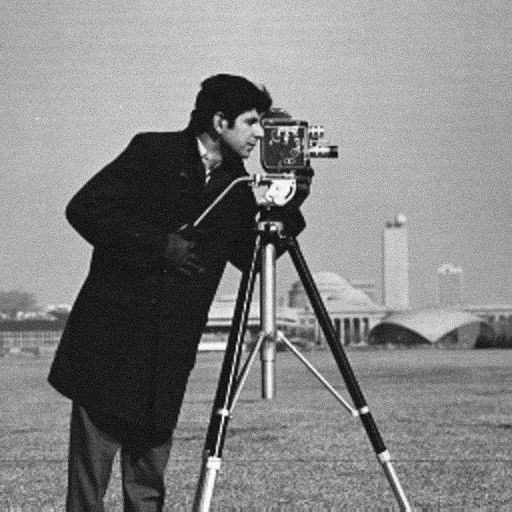} } \quad
\subfloat[][ROF, $\N = 16 \times 16$ (PSNR: 25.35)]{ \includegraphics[height=3.8cm]{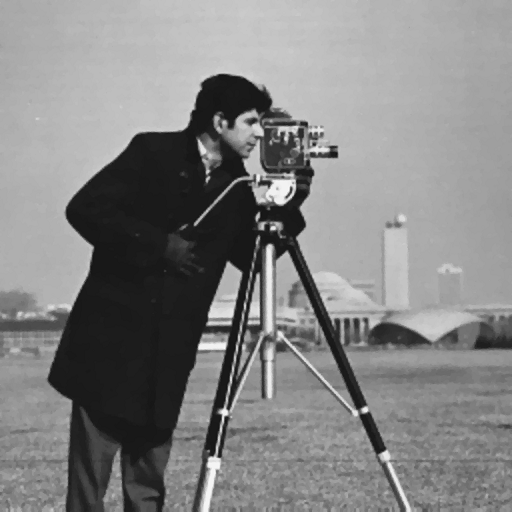} }
\caption{Test images and their results of Algorithm~\ref{Alg:ASM} applied to~\eqref{ROF_App} for $\N = 16 \times 16$ with $d/\delta = 2^6$.}
\label{Fig:test}
\end{figure}

In this section, we introduce several applications of the proposed method.
We also provide numerical experiments which support our theoretical results presented above.

All algorithms were implemented in C with MPI and performed on a computer cluster composed of seven machines, where each machine is equipped with two Intel Xeon SP-6148 CPUs~(2.4GHz, 20C) and 192GB~RAM.
Two test images ``Peppers $512 \times 512$'' and ``Cameraman $2048 \times 2048$'' that we used in our experiments are displayed in Figures~\ref{Fig:test}(a) and~(d).
As a measurement of the quality of image restoration, we provide the PSNR~(peak signal-to-noise ratio); the PSNR of a corrupted image $u \in X_h$ with respect to the original clean image $u_{\mathrm{orig}} \in X_h$ is defined by
\begin{equation*}
\mathrm{PSNR}(u) = 10 \log_{10} \left( \frac{\mathrm{MAX}^2 |\Omega|}{\| u - u_{\mathrm{orig}} \|^2}\right),
\end{equation*}
where $\mathrm{MAX} = 1$ is the maximum possible pixel value of the image.
In the following, we take the side length of elements $h=1$ and denote the side length of $\Omega$ by $d$, i.e., $d = |\Omega|^{1/2}$ for square images such as Figure~\ref{Fig:test}.
The scaled energy error $\alpha (\F(\p^{(n)}) - \F(\p^*))$ of the $n$th iterate $\p^{(n)}$ is denoted by $\zeta_n$, where the minimum energy $\F(\p^* )$ was computed by $10^6$ iterations of FISTA~\cite{BT:2009}.

\subsection{The Rudin--Osher--Fatemi model}
\begin{figure}[]
\centering
\subfloat[][Peppers $512 \times 512$, $\log$-$\log$ plot]{ \includegraphics[height=4.2cm]{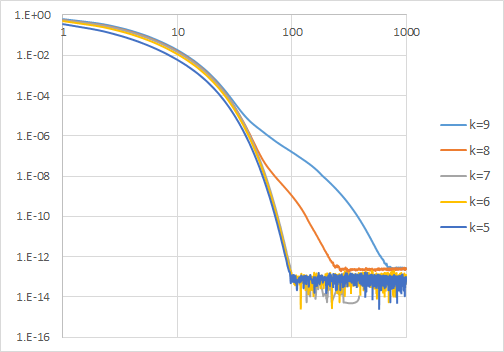} } \quad
\subfloat[][Peppers $512 \times 512$, normal-$\log$ plot]{ \includegraphics[height=4.2cm]{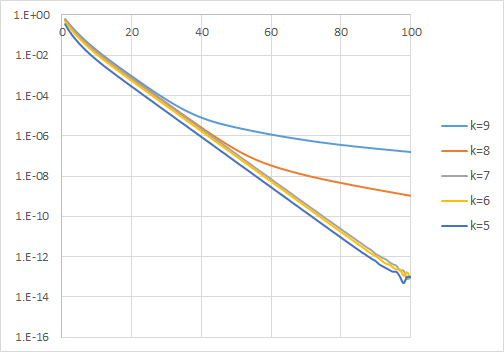} } \\
\subfloat[][Cameraman $2048 \times 2048$, $\log$-$\log$ plot]{ \includegraphics[height=4.2cm]{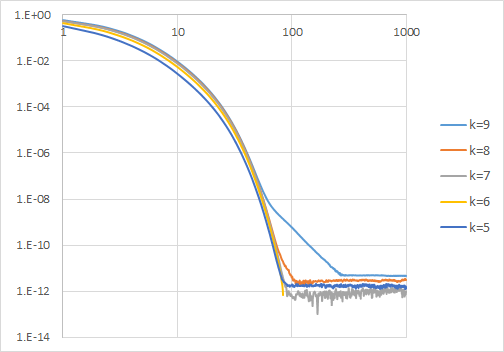} } \quad
\subfloat[][Cameraman $2048 \times 2048$, normal-$\log$ plot]{ \includegraphics[height=4.2cm]{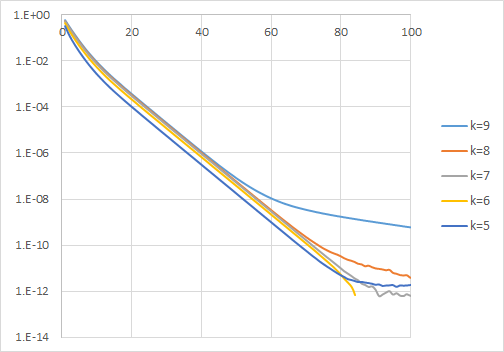} }
\caption{Decay of the relative energy error $\zeta_n / \zeta_0$ of Algorithm~\ref{Alg:ASM} applied to~\eqref{dual_ROF_App} for $d/\delta = 2^k$~($k=5,6,\dots,9$) with $\N = 8 \times 8$.}
\label{Fig:delta_ROF}
\end{figure}

\begin{figure}[]
\centering
\subfloat[][Peppers $512 \times 512$, $\log$-$\log$ plot]{ \includegraphics[height=4.2cm]{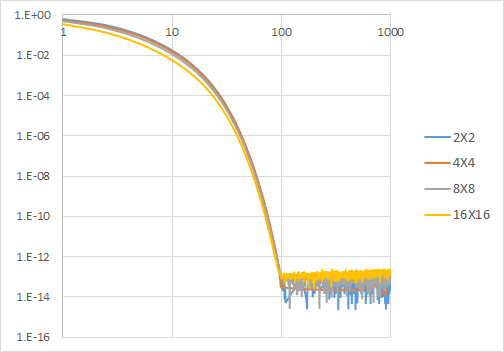} } \quad
\subfloat[][Peppers $512 \times 512$, normal-$\log$ plot]{ \includegraphics[height=4.2cm]{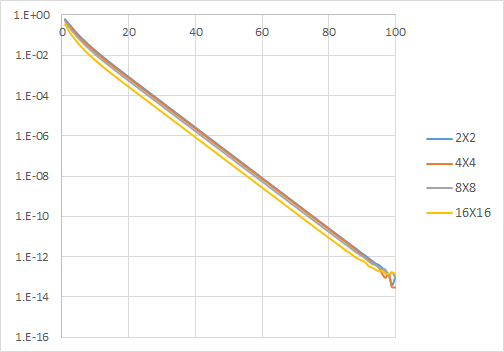} } \\
\subfloat[][Cameraman $2048 \times 2048$, $\log$-$\log$ plot]{ \includegraphics[height=4.2cm]{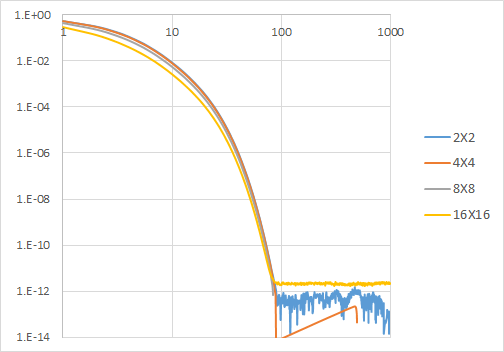} } \quad
\subfloat[][Cameraman $2048 \times 2048$, normal-$\log$ plot]{ \includegraphics[height=4.2cm]{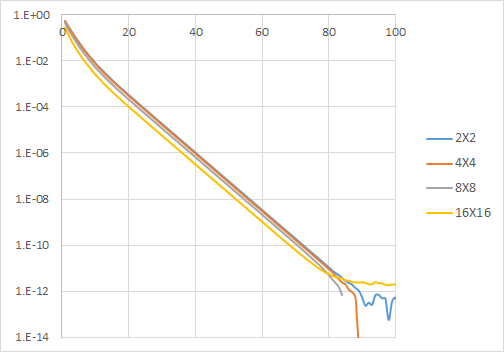} }
\caption{Decay of the relative energy error $\zeta_n / \zeta_0$ of Algorithm~\ref{Alg:ASM} applied to~\eqref{dual_ROF_App} for $\N = 2\times 2, \dots, 16 \times 16$ with $d/\delta = 2^6$.}
\label{Fig:N_ROF}
\end{figure}

The Fenchel--Rockafellar dual problem of the discrete ROF model~\eqref{ROF_App} is stated as
\begin{equation}
\label{dual_ROF_App}
\min_{\p \in C} \left\{ \F (\p) := \frac{1}{2\lambda} \| \div \p + \lambda f \|^2 \right\},
\end{equation}
and one can obtain the Frech\'{e}t derivative of $\F$ as
\begin{equation*}
\F'(\p) = \frac{1}{\lambda} \div^* \left( \div \p + \lambda f \right).
\end{equation*}
The projection onto $C$ can be easily computed by the pointwise Euclidean projection~\cite{LPP:2019}.
Therefore,~\eqref{dual_ROF_App} can be solved efficiently by, e.g., FISTA~\cite{BT:2009}.
Note that for the case of~\eqref{dual_ROF_App}, the primal-dual relation~\eqref{pd_equiv} reduces to the following:
\begin{equation*}
u^* = f + \frac{1}{\lambda} \div \p^* ,
\end{equation*}
where $u^* \in X_h$ solves~\eqref{ROF_App}.

For our experiments, test images shown in Figures~\ref{Fig:test}(a) and~(d) were corrupted by additive Gaussian noise with mean 0 and variance 0.05; see Figures~\ref{Fig:test}(b) and~(e).
The parameter $\lambda$ in~\eqref{ROF_App} was chosen by $\lambda = 10$.
In Algorithm~\ref{Alg:ASM}, we set $\tau = 1/4$.
Local problems in $\Omega_s'$, $s=1, \dots, \N$, were solved by FISTA~\cite{BT:2009} with $L = 8/\lambda$ and the stop criterion
\begin{equation}
\label{stop}
\frac{\| \div (\r_s^{(n+1)} - \r_s^{(n)}) \|^2}{|\Omega_s'|} \leq 10^{-18} \quad\textrm{or}\quad n = 1000.
\end{equation}
We note that the parameter selection $L = 8/\lambda$ is due to Proposition~\ref{Prop:inverse}.
The image results for the case $16\times 16$ are given in Figures~\ref{Fig:test}(c) and~(f), and they show no trace on the subdomain boundaries.

First, we observe how the convergence rate of Algorithm~\ref{Alg:ASM} is affected by $d/\delta$.
Figure~\ref{Fig:delta_ROF} shows the decay of the relative energy error $\zeta_n / \zeta_0$ for $d/\delta = 2^k$~($k=5,6,\dots,9$) when the number of subdomains is fixed by $\N = 8 \times 8$.
As Figures~\ref{Fig:delta_ROF}(a) and~(c) show, the threshold of the pseudo-linear convergence decreases as $d/\delta$ decreases, which verifies Corollary~\ref{Cor:1L}.
Furthermore, in the cases when $d/\delta \leq 2^7$, the threshold is so small that the behavior of Algorithm~\ref{Alg:ASM} is like linearly convergent algorithms.
Thus, the proposed method is as efficient as linearly convergent algorithms in practice.
We also observe from Figures~\ref{Fig:delta_ROF}(b) and~(d) that the linear convergence rate of Algorithm~\ref{Alg:ASM} is independent of $\delta$ as noted in Corollary~\ref{Cor:1L}.

Next, we consider the performance of the proposed DDM with respect to the number of subdomains $\N$.
Figure~\ref{Fig:N_ROF} shows the decay of $\zeta_n / \zeta_0$ when $\N$ varies from $2\times 2$ to $16 \times 16$ with $d/\delta = 2^6$.
We readily see that the convergence behavior of Algorithm~\ref{Alg:ASM} is almost the same regardless of $\N$.
Hence, we conclude that the convergence rate of Algorithm~\ref{Alg:ASM} does not depend on $\N$.

\begin{figure}[]
\centering
\subfloat[][Peppers $512 \times 512$]{ \includegraphics[height=4.2cm]{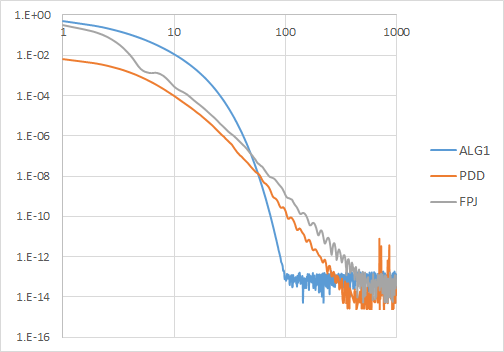} } \quad
\subfloat[][Cameraman $2048 \times 2048$]{ \includegraphics[height=4.2cm]{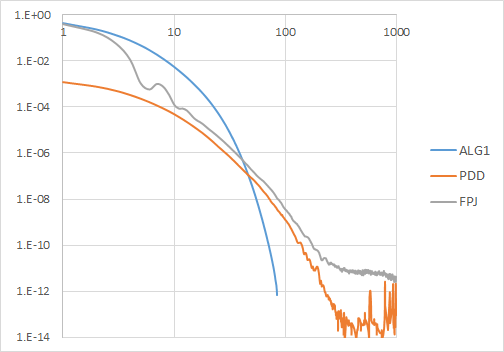} } 
\caption{Decay of the relative energy error $\zeta_n / \zeta_0$ of various DDMs applied to~\eqref{dual_ROF_App}, $\N = 8\times 8$.}
\label{Fig:comp}
\end{figure}

Finally, we compare the convergence behavior of the proposed method with two recently developed DDMs for the ROF model~\cite{LPP:2019,LP:2019b}.
The following algorithms were used in our experiments:
\begin{itemize}
\item ALG1: Algorithm~\ref{Alg:ASM}, $\N = 8 \times 8$, $d/\delta = 2^6$.
\item PDD: Primal DDM proposed in~\cite{LPP:2019}, $\N = 8 \times 8$, $L = 4$.
\item FPJ: Fast pre-relaxed block Jacobi method proposed in~\cite{LP:2019b}, $\N = 8 \times 8$.
\end{itemize}
As shown in Figure~\ref{Fig:comp}, as the number of iterations increases, the convergence rate of ALG1 becomes eventually faster than those of PDD and FPJ, which were proven to be $O(1/n^2)$-convergent.
We note that numerical results that verify the superior convergence properties of PDD and FPJ compared to existing $O(1/n)$-convergent DDMs were presented in~\cite{LPP:2019,LP:2019b}.

\begin{remark}
\label{Rem:FDM}
Even though the proposed method is based on finite element discretizations, a direct comparison with methods based on finite difference discretizations such as FPJ is possible in virtue of the equivalence relation presented in~\cite[Theorem~2.3]{LPP:2019}. 
\end{remark}

\subsection{The $TV$-$H^{-1}$ model}
\begin{figure}[]
\centering
\subfloat[][Peppers $512 \times 512$, $\log$-$\log$ plot]{ \includegraphics[height=4.2cm]{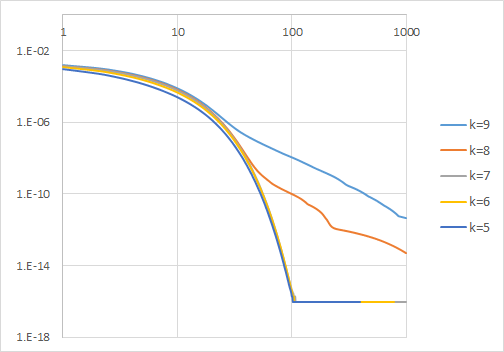} } \quad
\subfloat[][Peppers $512 \times 512$, normal-$\log$ plot]{ \includegraphics[height=4.2cm]{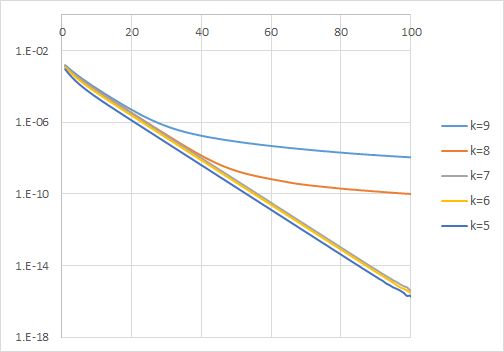} } \\
\subfloat[][Cameraman $2048 \times 2048$, $\log$-$\log$ plot]{ \includegraphics[height=4.2cm]{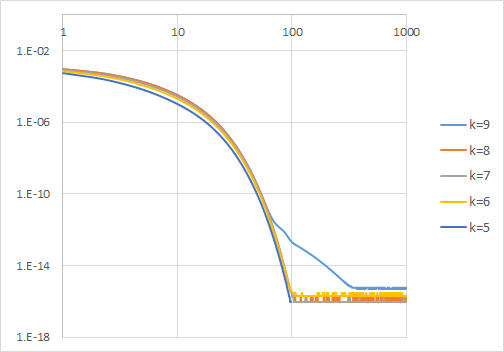} } \quad
\subfloat[][Cameraman $2048 \times 2048$, normal-$\log$ plot]{ \includegraphics[height=4.2cm]{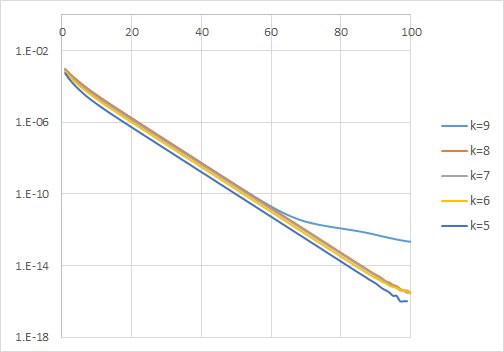} }
\caption{Decay of the relative energy error $\zeta_n / \zeta_0$ of Algorithm~\ref{Alg:ASM} applied to~\eqref{dual_K_App} for $d/\delta = 2^k$~($k=5,6,\dots,9$) with $\N = 8 \times 8$.}
\label{Fig:delta_K}
\end{figure}

\begin{figure}[]
\centering
\subfloat[][Peppers $512 \times 512$, $\log$-$\log$ plot]{ \includegraphics[height=4.2cm]{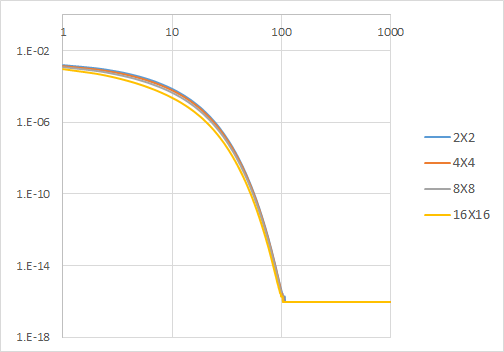} } \quad
\subfloat[][Peppers $512 \times 512$, normal-$\log$ plot]{ \includegraphics[height=4.2cm]{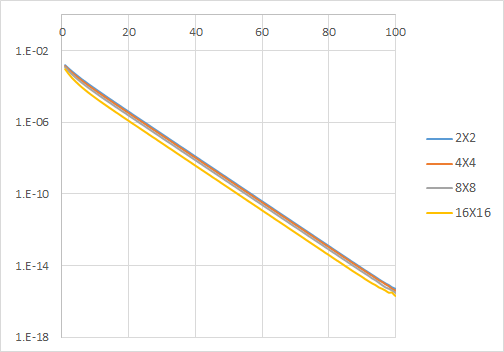} } \\
\subfloat[][Cameraman $2048 \times 2048$, $\log$-$\log$ plot]{ \includegraphics[height=4.2cm]{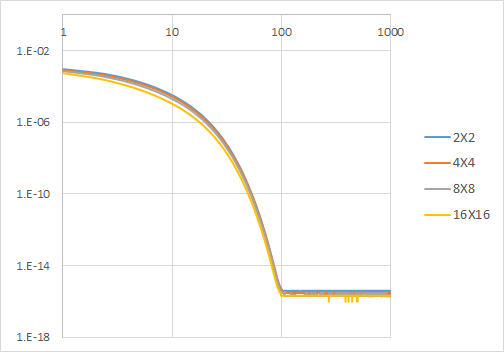} } \quad
\subfloat[][Cameraman $2048 \times 2048$, normal-$\log$ plot]{ \includegraphics[height=4.2cm]{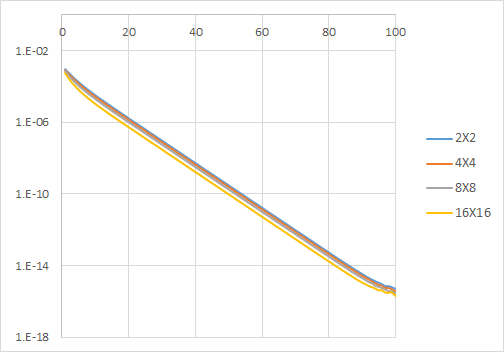} }
\caption{Decay of the relative energy error $\zeta_n / \zeta_0$ of Algorithm~\ref{Alg:ASM} applied to~\eqref{dual_K_App} for $\N = 2\times 2, \dots, 16 \times 16$ with $d/\delta = 2^6$.}
\label{Fig:N_K}
\end{figure}

Now, we consider the discrete $TV$-$H^{-1}$ model~\eqref{K_App}.
The dual problem of~\eqref{K_App} is presented as
\begin{equation}
\label{dual_K_App}
\min_{\p \in C} \left\{ \F (\p) := \frac{1}{2\lambda} \| \div \p \|_K^2 + \left< f, \div \p \right> \right\},
\end{equation}
where $\| v \|_K = \left< Kv, v \right>^{1/2}$ for $v \in X_h$.
The Frech\'{e}t derivative $\F' (\p)$ can be easily computed by
\begin{equation*}
\F' (\p) = \frac{1}{\lambda} \div^* (K \div \p + \lambda f).
\end{equation*}
If we have a solution $\p^* \in Y_h$ of~\eqref{dual_K_App}, then a solution $u^* \in X_h$ of~\eqref{K_App} can be obtained by
\begin{equation*}
u^* = f + \frac{1}{\lambda}K \div \p^* .
\end{equation*}

Now, we present the numerical results of Algorithm~\ref{Alg:ASM} for~\eqref{dual_K_App}.
The corrupted test images Figures~\ref{Fig:test}(b) and (e) are used as $f$ in~\eqref{dual_K_App}.
We set $\lambda = 10$.
In Algorithm~\ref{Alg:ASM}, the parameter $\tau$ is chosen by $\tau = 1/4$ and local problems in $\Omega_s'$, $s=1, \dots, \N$, were solved by FISTA~\cite{BT:2009} with $L=64/\lambda$ and the stop criterion~\eqref{stop}.
The parameter selection $L=64/\lambda$ is derived by Proposition~\ref{Prop:inverse} and the Gershgorin circle theorem for $K$~\cite{LeVeque:2007}.

Figure~\ref{Fig:delta_K} shows the decay of the relative energy error $\zeta_n / \zeta_0$ for various values of $d/\delta$ when $\N = 8 \times 8$.
We observe the same dependency of the convergence rate on $d/\delta$ as the ROF case: Algorithm~\ref{Alg:ASM} behaves as a linearly convergent algorithm if $d/\delta \leq 2^7$, and the rate of linear convergence is independent of $\delta$.
As Figure~\ref{Fig:N_K} shows, the dependency of the convergence rate of Algorithm~\ref{Alg:ASM} is independent of $\N$; the convergence rates when $\N = 2\times 2 , \cdots, 16 \times 16$ are almost the same.
In conclusion, Corollary~\ref{Cor:1L} is verified for the $TV$-$H^{-1}$ model, as well as the ROF model.

It is interesting to observe that the pseudo-linear convergence of Algorithm~\ref{Alg:ASM} is not contaminated even in the case of the large condition number $\kappa$.
While the condition number of~\eqref{K_App} is much larger than the one of~\eqref{ROF_App} in general, the pseudo-linear convergence is evident for both problems.
The threshold of the pseudo-linear convergence presented in Theorem~\ref{Thm:pseudo} has an upper bound independent of $\kappa$ as follows:
\begin{equation*}
\frac{4c_2 |\Omega|}{\kappa \sqrt{c_1 (c_1 + \kappa^{-2})}} \leq \frac{4c_2 |\Omega|}{c_1}.
\end{equation*}
Therefore, one can conclude that this observation is indeed reflected in Theorem~\ref{Thm:pseudo}.

\section{Conclusion}
\label{Sec:Conclusion}
We proposed an additive Schwarz method based on an overlapping domain decomposition for total variation minimization.
Differently from the existing work~\cite{CTWY:2015}, we showed that our method is applicable to not only the ROF model but also more general total variation minimization problems.
A novel technique using a descent rule for the convergence analysis of the additive Schwarz method was presented.
With this technique, we obtained the convergence rate of the proposed method as well as the dependency of the rate on the condition number of the model problem.
In addition, we showed the pseudo-linear convergence property of the proposed method, in which the convergence behavior of the proposed method is like linearly convergent algorithms if the overlapping width $\delta$ is large.
Numerical experiments verified our theoretical results.

Recently, the acceleration technique proposed in~\cite{BT:2009} was successfully applied to nonoverlapping DDMs for the ROF model and accelerated methods were developed~\cite{LPP:2019,LP:2019b}.
However, it is still open that how to adopt the acceleration technique to overlapping DDMs for general total variation minimization.

As a final remark, we note that the convergence analysis in this paper can be easily applied  to either a continuous setting or a finite difference discretization with slight modification.

\section*{Acknowledgments}
The author would like to thank Professor Chang-Ock Lee for insightful discussions and comments.

\bibliographystyle{siam}
\bibliography{refs_Schwarz_ROF}

\end{document}